%% file: ExVar.tex
\documentclass[12pt]{amsart}
\usepackage{geometry}                
\usepackage[parfill]{parskip}    
\usepackage{graphicx}
\usepackage{amssymb}

\usepackage[colorlinks,  urlcolor=blue]{hyperref}   
\usepackage{algorithm}
\usepackage{algpseudocode}
\usepackage{multirow}
\usepackage{caption}
\usepackage{subcaption}
\usepackage{adjustbox}

\newtheorem{thm}{Theorem}[section]
\newtheorem{lem}[thm]{Lemma}
\newtheorem{defn}[thm]{Definition}
\newtheorem{corollary}[thm]{Corollary}

\theoremstyle{definition}
\newtheorem{ass}[thm]{Assumption}
\newtheorem{example}[thm]{Example}

\theoremstyle{remark}

\DeclareSymbolFont{AMSb}{U}{msb}{m}{n}
\DeclareMathSymbol{\N}{\mathbin}{AMSb}{"4E}
\DeclareMathSymbol{\Z}{\mathbin}{AMSb}{"5A}
\DeclareMathSymbol{\R}{\mathbin}{AMSb}{"52}
\DeclareMathSymbol{\Q}{\mathbin}{AMSb}{"51}
\DeclareMathSymbol{\I}{\mathbin}{AMSb}{"49}
\DeclareMathSymbol{\C}{\mathbin}{AMSb}{"43}

\title{Diffusion Parameter Estimation for the Homogenized Equation}
\author{Theodoros Manikas and Anastasia Papavasiliou}
\address{Theodoros Manikas \\ Department of Statistics \\ University of Warwick \\ Coventry \\ CV4 7AL \\UK}
\address{Anastasia Papavasiliou \\ Department of Statistics \\ University of Warwick \\ Coventry \\ CV4 7AL \\UK}

\date{\today}                                           
\thanks{AP is supported by the Leverhulme RPG-2013-270: ``Statistical Inference of Complex Systems through Rough Paths''}

\begin{document}
\maketitle

\begin{abstract}
We construct a novel estimator for the diffusion coefficient of the limiting homogenized equation, when observing the slow dynamics of a multiscale model, in the case when the slow dynamics are of bounded variation. Previous research suggests subsampling the data on fixed intervals and computing the corresponding quadratic variation (see, for example, \cite{pavliotis2007parameter}). However, to achieve optimality, this approach requires knowledge of scale separation variable $\epsilon$. Instead, we suggest computing the quadratic variation corresponding to the local extrema of the slow process. Our approach results to a natural subsampling and avoids the issue of choosing a subsampling rate. We prove that the estimator is asymptotically unbiased and we numerically demonstrate that its $L_2$-error is smaller than the one achieved in \cite{pavliotis2007parameter}.
\end{abstract}

\input{Introduction} 

\input{Setting} 

\input{OU} 

\input{General} 
\input{numerics} 

\input{References}
\end{document}

%% file: Introduction.tex
\section{Introduction}
\label{chap: intro}

It is often the case that the most accurate models for describing the dynamics of physical or human-driven activity are multiscale in nature. For example, high-frequency financial data often exhibits multiscale characteristics in the sense that disparate structural features are associated with different time scales.  These features are usually described by the term market microstructure noise, which contains all different types of market inconsistencies such as non-synchronous trading and bid-ask spread.  In \cite{tsay2005analysis}, the author describes each of these effects and gives a comprehensive review. Processes exhibiting multiscale characteristics  also appear in other application areas, such as molecular dynamics \cite{schlick2010molecular}, atmospheric sciences or oceanography (see, for example,  \cite{majda2001mathematical,majda2006stochastic} and \cite{katsoulakis2004multiscale,katsoulakis2005multiscale}) and network traffic data \cite{abry2002multiscale}.

Finding a coarse-grained model that can effectively describe the dynamics of the initial multiscale model is an important problem and a highly active research area in applied mathematics. This is mainly due to the fact that such models are much more efficient to use in practice. Once the coarse-grained model has been extracted, the corresponding free parameters need to be estimated by fitting the model to the data. In this framework, the problem that one is confronted with is the mismatch between the coarse-grained model and the data generated by the full multiscale system.

The parameter estimation problem in the context of multiscale diffusions can be separated into four different cases, depending on (i) whether the limiting equation is the averaging or homogenization limit and (ii) whether we are interested in estimating the drift or the diffusion coefficient of the limiting equation. In \cite{pavliotis2007parameter}, the authors study all problems but for particular types of diffusions where we can get closed form estimators for the unknown parameters of the limiting equation. In \cite{papavasiliou2009maximum}, the authors discuss the problem in a general context but only for drift estimation. In both papers, the authors suggest using the Maximum Likelihood estimators corresponding to the limiting model with observation of the slow variable that have been sufficiently subsampled. This methodology has been applied to molecular dynamics (see \cite{pavliotis2008parameter}) and high-frequency data (see \cite{sykulski2008multiscale}).

In this paper, we are concerned with the estimation of the diffusion coefficient of the homogenization limit. This has be addressed in \cite{pavliotis2007parameter}, where the authors show that the quadratic variation on subsampled data at intervals of size $\delta$ converges in law to the limiting diffusion coefficient as $\epsilon\to0$, provided that $\delta = \epsilon^\alpha$, for $\alpha\in(0,1)$. They also show that the $L_2$-error is minimized for $\alpha=\frac{1}{2}$ and is of order ${\mathcal O}(\epsilon^\frac{2}{3})$. Some further discussion of these estimators and their properties can be found in \cite{ZhangPapavasiliou}. However, the scale separation variable $\epsilon$ is not known, so one cannot be sure that a chosen $\delta$ leads to the right result, let alone choose the optimal subsampling rate. 

This problem has also been addressed in \cite{papavasiliou2011coarse}. The estimator proposed there is the total $p$\textendash variation, defined as the supremum of finite quadratic sums over all possible partitions. It was shown that the estimator is asymptotically unbiased and its $L_2$-error is of order ${\mathcal O}(\epsilon)$, thus improving the $L_2$-error of \cite{pavliotis2007parameter}, while avoiding the issues related to choosing a subsampling rate. However, due to the technical difficulties related to dealing with the total $p$-variation, the paper only discusses the diffusion estimation problem for a mutliscale Ornstein-Uhlenbeck process, with the slow dynamics being of bounded variation (the slow dynamics are often called `natural Brownian motion')

In this paper, we build upon ideas in \cite{papavasiliou2011coarse}. Our estimator, however, is simpler to work with as it does not involve a supremum, making it both more practical and easier to work with. Moreover, in \cite{papavasiliou2011coarse} the author assumes continuous observation of the slow process while we make the more realistic assumption of discrete observations. We consider mutliscale diffusions whose slow dynamics are of bounded variation and the diffusion coefficient of the homogenized is constant.

The remainder of the paper is organised as follows. In section \ref{sec: setting}, we describe the precise class of models and we state the statistical inference question that we are interested in. We, then, suggest a novel estimator. In section \ref{sec: OU}, we prove that our estimator is asymptotically unbiased in the case of the Ornstein-Uhlenbeck (OU) model within the context of the general family of models that we consider. In section \ref{sec: general}, we study the properties of the estimator for a more general class of models than the OU. Finally, in section \ref{sec: numerical_results}, we present a numerical investigation of the properties of our estimator. In particular, we investigate the behaviour of the $L_2$ error and we numerically demonstrate that it is of order ${\mathcal O}(\epsilon)$.

%% file: Setting.tex
\section{Setting}
\label{sec: setting}

We consider the following system of stochastic differential equations:
\begin{subequations}
\begin{align}
dx^\epsilon_t & = \frac{1}{\epsilon}f(y^\epsilon_t)dt,\label{eq:general_slow_variable}\\
dy^\epsilon_t & = \frac{1}{\epsilon^{2}}g(y^\epsilon_t)dt+\frac{\beta(y^\epsilon_t)}{\epsilon}dV_t,\label{eq:general_fast_variable}
\end{align}
\label{eq:general_model}\end{subequations}
where $V$ is standard Brownian motion, defined on the filtered probability space $(\Omega, \{{\mathcal F_t}\}_{t>0}, {\mathbb P})$. We are interested in the case where the `slow' component of the process, $x^\epsilon$, converges in distribution to the solution of a stochastic differential equation
\begin{equation}
dX_t=\sigma dW_t,\qquad X(0)=x_{0},\label{eq:Homo_general_model}
\end{equation}
where $\sigma$ is a constant depending on $f,g$ and $\beta$ and $W$ is also standard Brownian motion. This convergence holds under appropriate assumptions to be discussed later (see \cite{pavliotis2008multiscale}). We call this equation `the homogenized equation' and its solution $X$ the `homogenization limit'. 

In this paper, we restrict our study to the case where both $x^\epsilon$ and its limit $X$ are one-dimensional processes. Our goal is to estimate the diffusion coefficient $\sigma^2$ of the homogenized equation from discrete observations of the slow process $x^\epsilon$. More precisely, let us assume that we observe $\{x^\epsilon_{t_i}, i=0,\dots,n\}$, for $t_i = i\delta$ and $n\delta = T$. We want to construct an estimator for the diffusion coefficient $\sigma^2$, such that as $n\to\infty$, the approximation error is of order ${\mathcal O}(\epsilon)$. Note that $\epsilon$ is a fixed variable that is inherent to the process and we have no control over while we have some control over $n$ (how often to sample or for how long). For the rest of the paper, however, we will assume that $T$ is fixed and equal to $T=1$, so $n\to\infty$ is equivalent to $\delta\to 0$. Note that while the grid points $t_i$ depend on $\delta$ (or, equivalently, $n$), to ease the notation we will not explicitly show this dependence unless there is a risk of ambiguity. 

If we observed the homogenized equation $X$ rather than the slow process $x^\epsilon$, then it is well known that $\sigma^2$ can be efficiently estimated from the normalised Quadratic Variation, appropriately discretized to only depend on the discrete observations, i.e.
\begin{equation}
D_{2}\left(X\right)_n=\sum_{i=1}^n(\Delta X_{t_i})^{2}
\end{equation}
where $\Delta X_{t_i}:=X_{t_i}-X_{t_{i-1}}$ for $t_i\in{\mathcal D}_n = \{ k\delta, k=0,\dots,n\}$. Then, we know that $D_{2}\left(X\right)_n\to \sigma^2$ almost surely, as $n\to\infty$. However, if instead we use the bounded variation process $x^\epsilon$, the corresponding quadratic variation $D_{2}\left(x^\epsilon\right)_n$ converges a.s. to $0$. Thus, because of the mismatch between model and data, the standard estimator is not longer useful. 

To avoid this problem, in \cite{papavasiliou2011coarse}, the author suggests using the total $2$-variation of the process (see \cite{lyons2002system}) rather than the quadratic variation, defined as 
\begin{equation}
D_{2}^{\text{Total}}\left(x^\epsilon\right)_{n}=\sup_{\mathcal{D}([0,T])}\left(\sum_{\tau_{i}\in\mathcal{D}([0,T])}(x^\epsilon_{\tau_{i}}-x^\epsilon_{\tau_{i-1}})^{2}
\right),
\end{equation}
where the supremum is over all finite partitions of $[0,T]$. By taking the supremum over all finite partitions, the total $2$-variation can only be zero if the path $x^\epsilon$ is constant, so it maintains much more information than quadratic variation. In the case of a piecewise linear path, the supremum is achieved at a subset of the extremal points of the path \cite{Driver13}. However, this is still computationally very inefficient and a cause of technical difficulties, which is what limited \cite{papavasiliou2011coarse} to the analysis to the mutliscale Ornstein-Uhlenbeck model. Moreover, this estimator assumes continuous observation of the slow process $x^\epsilon$, which is unrealistic. 

In this paper, we construct a novel estimator that we will call the Extrema Quadratic Variation, whose construction is based on a simplification of the total $2$-variation. First, we approximate the slow process $x^\epsilon$ by its linear interpolation on the observations that we will denote by $x^\epsilon(n)$, so that the estimator only depends on available data. Then, instead of computing the total $2$-variation by identifying the subset of the extremal points where the supremum is achieved, we consider all extremal points. More precisely, the Extrema Quadratic Variation is defined as follows:
\begin{defn}
\label{subsec: ExtQV}
Let $x:[0,T]\rightarrow\mathbb{R}$ be a real\textendash valued continuous path and let $x(n)$ be the piecewise linear interpolation of $x$ on the homogeneous grid ${\mathcal D}_n = \{i\delta, i=0,\dots,n\}$ with $n\delta = T$. We define the {\it Extrema Quadratic Variation} (ExtQV) of the path on grid ${\mathcal D}_n$ as
\begin{equation}
D_{2}^{\text{Ext}}\left(x\right)_n= \sum_{\tau_{i}\in\mathcal{E}_{n}([0,T])}(x_{\tau_{i}}-x_{\tau_{i-1}})^{2},
\end{equation}
where $\mathcal{E}_{n}([0,T])=\left\{0= \tau_{0},\tau_{1},...,\tau_{k}=T\right\}$ is the set of local extremal points of $x(n)$. We say that a point $t_{i}$ in ${\mathcal D}_n$ is an extremal point and we write $t_{i}\in\mathcal{E}_{n}([0,T])$ if $\Delta x_{t_{i}}\Delta x_{t_{i+1}}=\left(x_{t_{i}}-x_{t_{i-1}}\right)\left(x_{t_{i+1}}-x_{t_{i}}\right)<0$.
\end{defn}

The computation of the quadratic variation requires the consideration of all the increments of the original path whereas the extrema quadratic variation only considers the increments of the extremal path. Note that as $\epsilon$ gets smaller, the process $x^\epsilon$ gets closer to $X$, which is a process of finite Quadratic Variation and thus, we expect the number of local extremal points to increase. Thus, the Extrema Quadratic Variation provides a natural subsampling of the process which depends on unknown $\epsilon$. This is why we expect it to outperform the Quadratic Variation estimator on the subsampled process suggested in \cite{pavliotis2007parameter}.

In Figure \ref{fig: OriginalVsExtrema} we illustrate graphically an example of an extremal path. The black line is the linear interpolation of the original path $x$ on the grin ${\mathcal D}_n$ and the red line is the corresponding extremal path.

\begin{figure}[H]
\begin{center}
\includegraphics[width=11cm,height=6.6cm]{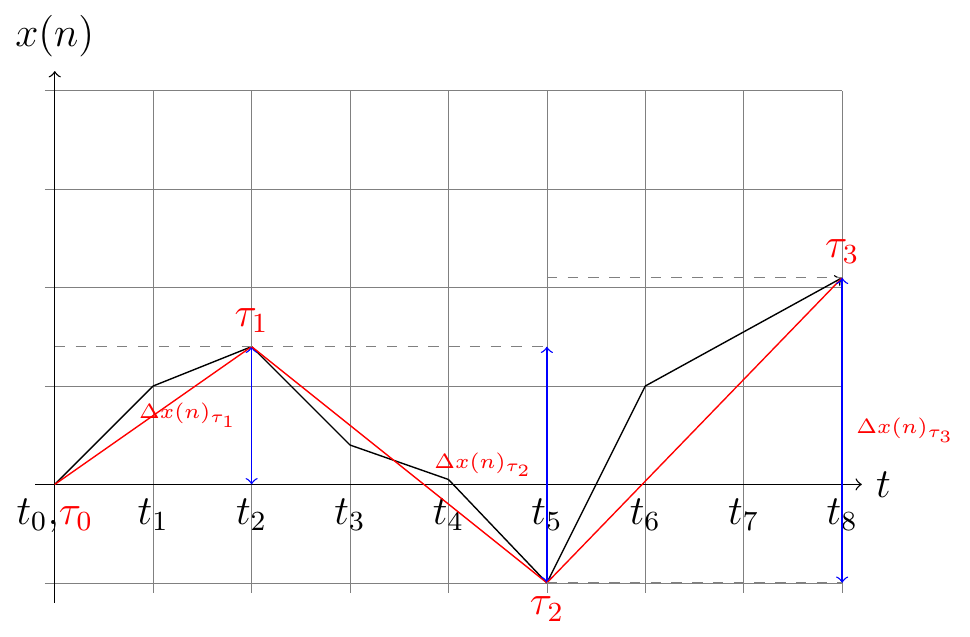}
\end{center}
\caption[Graphical representation of an extremal path.]{Graphical representation of an extremal path: the original path (black line) and the extremal path (red line).}
\label{fig: OriginalVsExtrema}
\end{figure}

An alternative way to compute $D_{2}^{\text{Ext}}\left(x\right)_n$ that appears to be very useful in the analytic computation of the expectation is presented below. By definition, $D_{2}^{\text{Ext}}\left(x\right)_n$ is the sum of squared returns of the original process plus two times the product of those increments such that the consecutive products of the increments between these two are all positive.
So, by expanding the squares, $D_{2}^{\text{Ext}}\left(x\right)_n$ can be written as:
\begin{align}
D_{2}^{\text{Ext}}(x)_n& =D_{2}(x)_n + 2\sum_{i=1}^{j-1}\sum_{j=2}^{n}\bigg(\Delta x_{t_{i}}\Delta x_{t_{j}}\prod_{k=j}^{i-1}\mathbf{1}_{\R_+}\left(\Delta x_{t_k} \Delta x_{t_{k+1}}\right)\bigg), \label{eq: ext_QV}
\end{align}
Note that the event of $\{ \Delta x_{t_k} \Delta x_{t_{k+1} >0, k = j,\dots, i-1\ }$ can be written as the union of events $\{ \Delta x_{t_k} >0, k = j,\dots, i\}$ and $\{ \Delta x_{t_k} <0, k = j,\dots, i\}$, i.e. increments having the same sign is the same as all increments being either positive or negative. This will allow us to simplify computations further.

%% file: OU.tex
\section{A toy example}

\label{sec: OU}

In order to build intuition about the behaviour of the Extrema Quadratic Variation estimator defined in \ref{subsec: ExtQV} , we start by studying its properties in the context of a system that is a spacial case  \eqref{eq:general_model} and also a special case or the Ornstein-Uhlenbeck model. More specifically, we consider the following model
\begin{subequations}
\begin{align}
dx^{\epsilon}_t & =\frac{\sigma}{\epsilon}y^{\epsilon}_tdt,\label{eq: slow_dyn_multi}\\
dy^{\epsilon}_t & =-\frac{1}{\epsilon^{2}}y^{\epsilon}_tdt+\frac{1}{\epsilon}dW_t,\label{eq: fast_dyn_multi}
\end{align}
\label{subeq: multiOU}
\end{subequations}
where $W$ denotes the standard one-dimensional Brownian motion defined on the filtered probability space $(\Omega, \{{\mathcal F_t}\}_{t>0}, {\mathbb P})$, $\sigma\in\mathbb{R}_{+}$ is a positive constant and $0<\epsilon<<1$ denotes a small parameter that controls the scale separation. The fast dynamics are described by an Ornstein-Uhlenbeck process whose invariant distribution is the Gaussian distribution ${\mathcal N}(0,\frac{1}{2})$. We will assume that $y_0$ is also a random variable with the invariant distribution ${\mathcal N}(0,\frac{1}{2})$, so that process $y^\epsilon$ is stationary.

It is easy to see that the slow process $x^\epsilon$ can be equivalently expressed as the solution of the following SDE
\begin{equation}
dx^{\epsilon}_t=\sigma\left(dW_t-\epsilon dy^{\epsilon}_t\right).
\label{eq: multi_OU_equivalent_expression}
\end{equation}
Therefore, allowing $\epsilon\rightarrow0$ we deduce that the corresponding homogenization limit is the solution to
\begin{equation}
dX_t=\sigma dW_t,\qquad\quad X_0=x_{0}.\label{eq: homo_SDE}
\end{equation}
In this case, the convergence holds pathwise in $L_2$ \cite{ZhangPapavasiliou}.

We will show that, in this case, the Extrema Quadratic Variation estimator is asymptotically unbiased. More precisely, we prove the following:

\begin{thm}\label{thm: main_theorem_OUmodel}
Let $x^{\epsilon}:[0,T]\rightarrow\mathbb{R}$ be a real\textendash valued
path described by Eq.\eqref{subeq: multiOU}. Then,
\begin{equation}
\lim_{\epsilon\rightarrow0}\lim_{n\rightarrow\infty}\mathbb{E}\left[D_{2}^{\text{Ext}}\left(x^{\epsilon}\right)_{n}\right]=\sigma^{2}.
\end{equation}
\end{thm}

Before proving the theorem, we prove the following lemma that allows us to write the increment of the slow process $x^\epsilon$ in terms of the fast process $y^\epsilon$.

\begin{lem}
\label{lemma: OU}
Let $(x^\epsilon, y^\epsilon)$ satisfy \eqref{subeq: multiOU}. Then, we can write
\begin{equation}
\Delta x^{\epsilon}_{t_{k}}:= x^\epsilon_{t_k}-x^\epsilon_{t_{k-1}}  =\sigma\left(\epsilon(1-e^\frac{\delta}{\epsilon^2})y^{\epsilon}(t_{k-1})-\int_{t_{k-1}}^{t_k}
\left(e^{-\frac{(t_{k}-u)}{\epsilon^2}}-1\right)dW_{u}\right).
\label{eq: slow_increment_OU}
\end{equation}
\end{lem}
\begin{proof}
First, using \eqref{eq: multi_OU_equivalent_expression}, we get
\[ \Delta x^{\epsilon}_{t_{k}} = \sigma\left(\Delta W_{t_k} - \epsilon \Delta y^\epsilon_{t_k}\right).\]
Using the known formula for the solution of \eqref{eq: fast_dyn_multi} (which can be easily verified using It\^o's formula), we write
\begin{equation}
y^{\epsilon}_{t_{k}}=e^{-\delta/\epsilon^{2}}y^{\epsilon}_{t_{k-1}}+\frac{1}{\epsilon}\int_{t_{k-1}}^{t_{k}}e^{-\frac{(t_{k}-u)}{\epsilon^{2}}}dW_u.
\label{eq:OU_sol}
\end{equation}
The result follows.
\end{proof}

\begin{corollary}
\label{cor: DX OU}
Let $(x^\epsilon, y^\epsilon)$ satisfy \eqref{subeq: multiOU}. Then, if $y^\epsilon_t$ is stationary, the sequence of increments $\{\Delta x^\epsilon_{t_i}\}_{i=1}^n$ is also stationary irrespectively of $x_0$. Moreover, the increments $\{\Delta x^\epsilon_{t_i}\}_{i=1}^n$ are mean zero Gaussian random variables.
\end{corollary}

\begin{proof}[Proof of Theorem \ref{thm: main_theorem_OUmodel}]
Using the expression in \eqref{eq: ext_QV}, we get
\begin{equation}
{\mathbb E}\left( D_{2}^{\text{Ext}}(x^{\epsilon})_{n}\right) = {\mathbb E}\left( D_{2}(x^{\epsilon})_{n}\right)+2\sum_{i=2}^{n}\sum_{j=1}^{i-1}{\mathbb E}\left(\Delta x^{\epsilon}_{t_{i}}\Delta x^{\epsilon}_{t_{j}}\prod_{k=j}^{i-1}\mathbf{1}_{\R_+}\left(\Delta x^{\epsilon}_{t_{k}}\Delta x^{\epsilon}_{t_{k+1}}\right)\right).
\label{eq: ext_QV_2}
\end{equation}
Since $x^\epsilon$ is a bounded variation process, it is not hard to show that
\[ \lim_{n\rightarrow\infty}{\mathbb E}\left( D_{2}(x^{\epsilon})_{n}\right) = 0.\]
Moreover, using corollary \ref{cor: DX OU}, i.e. the stationarity and symmetry of increments, we deduce that
\[ 2\sum_{i=2}^{n}\sum_{j=1}^{i-1}{\mathbb E}\left(\Delta x^{\epsilon}_{t_{i}}\Delta x^{\epsilon}_{t_{j}}\mathbf{1}_{C_n(i,j)}\left(x(n))\right)\right) = 4\sum_{k=2}^{n}(n+1-k)\mathbb{E}
\left( \Delta x^{\epsilon}_{t_{1}}\Delta x^{\epsilon}_{t_{k}}\prod_{j=1}^k\mathbf{1}_{\R_+}\left(\Delta x^{\epsilon}_{t_{j}}\right)\right)
\]
Thus, to prove the theorem we need to show that 
\begin{equation}
\label{eq: reduction1 OU} 
\lim_{\epsilon\rightarrow0}\lim_{n\rightarrow\infty}4\sum_{k=2}^{n}(n+1-k)\mathbb{E}
\left( \Delta x^{\epsilon}_{t_{1}}\Delta x^{\epsilon}_{t_{k}}\prod_{j=1}^k\mathbf{1}_{\R_+}\left(\Delta x^{\epsilon}_{t_{j}}\right)\right) = \sigma^2.
\end{equation}
Using lemma \ref{lemma: OU}, we write
\begin{eqnarray}
\Delta x^{\epsilon}_{t_{1}}\Delta x^{\epsilon}_{t_{k}} &=& \sigma^{2}\bigg(\epsilon^{2}(1-e^{-\frac{\delta}{\epsilon^2}})^{2}y^{\epsilon}_{t_{0}}y^{\epsilon}_{t_{k-1}}\nonumber \\
&  & \quad -\epsilon(1-e^{-\frac{\delta}{\epsilon^2}})y^{\epsilon}_{t_{0}}\int_{t_{k-1}}^{t_k}
\left(e^{-\frac{(t_{k}-u)}{\epsilon^2}}-1\right)dW_{u}\nonumber \\
&  & \qquad-\epsilon(1-e^{-\frac{\delta}{\epsilon^2}})y^{\epsilon}_{t_{k-1}}\int_{t_{0}}^{t_{1}}
\left(e^{-\frac{(t_{1}-u)}{\epsilon^2}}-1\right)dW_{u}\nonumber \\
& &  \quad\qquad + \int_{t_{0}}^{t_{1}}
\left(e^{-\frac{(t_{1}-u)}{\epsilon^2}}-1\right)dW_{u}\int_{t_{k-1}}^{t_k}
\left(e^{-\frac{(t_{k}-u)}{\epsilon^2}}-1\right)dW_{u} \bigg).
\label{eq: prod_Delta_x}
\end{eqnarray}
First, we show that 
\begin{equation} 
\label{eq: 2nd term OU}
\lim_{n\to\infty} 4\epsilon(1-e^{-\frac{\delta}{\epsilon^2}})\sum_{k=2}^{n}(n+1-k)\mathbb{E}
\left(y^{\epsilon}_{t_{0}}\int_{t_{k-1}}^{t_k}
\left(e^{-\frac{(t_{k}-u)}{\epsilon^2}}-1\right)dW_{u}\prod_{j=1}^k\mathbf{1}_{\R_+}\left(\Delta x^{\epsilon}_{t_{j}}\right)\right)=0,
\end{equation}
by obtaining an appropriate bound for the expectation. We write
\begin{eqnarray*}
&\mathbb{E}
\left(y^{\epsilon}_{t_{0}}\int_{t_{k-1}}^{t_k}
\left(e^{-\frac{(t_{k}-u)}{\epsilon^2}}-1\right)dW_{u}\prod_{j=1}^k\mathbf{1}_{\R_+}\left(\Delta x^{\epsilon}_{t_{j}}\right)\right) \leq&\\
&\leq \mathbb{E}
\left(y^{\epsilon}_{t_{0}}\int_{t_{k-1}}^{t_k} \left(e^{-\frac{(t_{k}-u)}{\epsilon^2}}-1\right)dW_{u}\right)\leq&\\
&\leq \mathbb{E}\left(\left(y^{\epsilon}_{t_{0}}\right)^2\right)^\frac{1}{2}\mathbb{E}\left( \left( \int_{t_{k-1}}^{t_k} \left(e^{-\frac{(t_{k}-u)}{\epsilon^2}}-1\right)dW_{u}\right)^2\right)^\frac{1}{2} =&\\
&= \frac{1}{\sqrt{2}}\mathbb{E}\left(  \int_{t_{k-1}}^{t_k} \left(e^{-\frac{(t_{k}-u)}{\epsilon^2}}-1\right)^2 du\right)^\frac{1}{2} = \frac{1}{\sqrt{2}}\left(\delta + \frac{1}{2}\epsilon^2(-3-e^{-\frac{2\delta}{\epsilon^2}}+4e^{-\frac{\delta}{\epsilon^2}})\right)^\frac{1}{2}\sim {\mathcal O}(\delta^\frac{3}{2}),&
\end{eqnarray*}
where the first inequality follows from the monotonicity of expectation, the second from Cauchy-Schwartz and the last equality from It\^o isometry and the stationarity of $y^\epsilon$. Noting that $\delta = \frac{1}{n}$, we can bound the sum in \eqref{eq: 2nd term OU} is bounded by
\[ C(\epsilon) \frac{1}{n}\sum_{k=2}^{n}(n+1-k)\frac{1}{n^\frac{3}{2}},\]
for some constant $C(\epsilon)$ that depends only on $\epsilon$. Taking $n\to\infty$, we get \eqref{eq: 2nd term OU}. In a similar manner, we can show that 
\begin{equation}
\label{eq: 3rd term OU}
\lim_{n\to\infty} 4\epsilon(1-e^{-\frac{\delta}{\epsilon^2}})\sum_{k=2}^{n}(n+1-k)\mathbb{E}
\left(y^{\epsilon}_{t_{k-1}}\int_{t_{0}}^{t_1}
\left(e^{-\frac{(t_{1}-u)}{\epsilon^2}}-1\right)dW_{u}\prod_{j=1}^k\mathbf{1}_{\R_+}\left(\Delta x^{\epsilon}_{t_{j}}\right)\right)=0,
\end{equation}
and
\begin{equation}
\label{eq: 4th term OU}
\lim_{n\to\infty} 4\sum_{k=2}^{n}(n+1-k)\mathbb{E}\left(
\int_{t_{0}}^{t_1}
\left(e^{-\frac{(t_{1}-u)}{\epsilon^2}}-1\right)dW_{u}\int_{t_{k-1}}^{t_k}
\left(e^{-\frac{(t_{k}-v)}{\epsilon^2}}-1\right)dW_{v}\prod_{j=1}^k\mathbf{1}_{\R_+}\left(\Delta x^{\epsilon}_{t_{j}}\right)\right)=0,
\end{equation}
Thus, from \eqref{eq: prod_Delta_x}, \eqref{eq: 2nd term OU}, \eqref{eq: 3rd term OU}, \eqref{eq: 4th term OU}, it follows that to show \eqref{eq: reduction1 OU}, we only need to show
\begin{equation}
\lim_{\epsilon\rightarrow0}\lim_{n\rightarrow\infty}4\epsilon^{2}(1-e^{-\frac{\delta}{\epsilon^2}})^{2}\sum_{k=2}^{n}(n+1-k)\mathbb{E}
\left(y^{\epsilon}_{t_{0}}y^{\epsilon}_{t_{k-1}} \prod_{j=1}^k\mathbf{1}_{\R_+}\left(\Delta x^{\epsilon}_{t_{j}}\right)\right) = 1
\label{eq: reduction2 OU}
\end{equation}
From lemma \ref{lemma: OU}, it follows that
\begin{equation}
\Delta x^{\epsilon}_{t_{j}}>0 \iff y_{t_{j-1}}>\frac{\int_{t_{j-1}}^{t_{j}}\left(e^{-\frac{(t_{j}-u)}{\epsilon^2}}-1\right)dW_{u}}{\epsilon(1-e^{-\frac{\delta}{\epsilon^2}})} := M_j.
\label{eq: M OU}
\end{equation}
With this notation, one can easily check that
\[ 
\mathbf{1}_{\R_+}\left(\Delta x^{\epsilon}_{t_{j}}\right) = \mathbf{1}_{\R_+}\left(y_{t_{j-1}}-M_j\right) \leq \mathbf{1}_{\R_+}\left(y_{t_{j-1}}\right) + \mathbf{1}_{\R_+}\left(|M_j|-|y_{t_{j-1}}|\right)
\]
Then, the expectation in the sum of \eqref{eq: reduction2 OU} can be bounded by
\begin{eqnarray}
&\mathbb{E}\left(y^{\epsilon}_{t_{0}}y^{\epsilon}_{t_{k-1}} \prod_{j=1}^k\mathbf{1}_{\R_+}\left(\Delta x^{\epsilon}_{t_{j}}\right)\right)  \leq&\nonumber\\
&\mathbb{E}\left(y^{\epsilon}_{t_{0}}y^{\epsilon}_{t_{k-1}} \prod_{j=1}^k\mathbf{1}_{\R_+}\left(y^{\epsilon}_{t_{j-1}}\right)\right) + \mathbb{E}
\left(y^{\epsilon}_{t_{0}}y^{\epsilon}_{t_{k-1}} \prod_{j=1}^k\mathbf{1}_{\R_+}\left(|M_j|-|y_{t_{j-1}}|\right)\right).&
\label{eq: term1 OU bound1}
\end{eqnarray}
Using Cauchy-Schwartz, the second expectation in the bound can be further bounded by
\begin{eqnarray}
&\mathbb{E}\left(y^{\epsilon}_{t_{0}}y^{\epsilon}_{t_{k-1}} \prod_{j=1}^k\mathbf{1}_{\R_+}\left(|M_j|-|y_{t_{j-1}}|\right)\right)&\nonumber\\
&\mathbb{E}\left((y^{\epsilon}_{t_{0}})^2\right)^\frac{1}{2}
\mathbb{E}\left((y^{\epsilon}_{t_{k-1}})^2 \prod_{j=1}^k\mathbf{1}_{\R_+}\left(|M_j|-|y_{t_{j-1}}|\right)\right)^\frac{1}{2} \leq \frac{1}{\sqrt{2}}\mathbb{E}\left((M_{k-1}^2 \right)^\frac{1}{2}&
\label{eq: term1 OU bound2}
\end{eqnarray}
Using upper bounds \eqref{eq: term1 OU bound1} and \eqref{eq: term1 OU bound2} in \eqref{eq: reduction2 OU}, we see that one of the terms converges to zero as in \eqref{eq: 2nd term OU}. Thus, it remains to show that
\begin{equation}
\lim_{\epsilon\rightarrow0}\lim_{n\rightarrow\infty}4\epsilon^{2}(1-e^{-\frac{\delta}{\epsilon^2}})^{2}\sum_{k=2}^{n}(n+1-k)\mathbb{E}
\left(y^{\epsilon}_{t_{0}}y^{\epsilon}_{t_{k-1}} \prod_{j=1}^k\mathbf{1}_{\R_+}\left(y^{\epsilon}_{t_{j-1}}\right)\right) = 1.
\label{eq: reduction3 OU}
\end{equation}

Let $\tau^{\delta,\epsilon}_{y_{t_0}}$ be the first time that the discretised process $\{y_{t_k}; k\in\N\}$ becomes negative, starting from $y_{t_0}$, i.e.
\begin{equation}
\label{first 0 crossing}
\tau^{\delta,\epsilon}_{y_{t_0}} = \min\{ k\in\N;  y_{t_k}<0\}.
\end{equation}
Then, given $y_{t_0} >0$,
\[ 
\prod_{j=1}^{k-1}\mathbf{1}_{\R_+}\left(y^{\epsilon}_{t_{j}}\right) = \mathbf{1}_{(t_{k-1},+\infty)}(\tau^{\delta,\epsilon}_{y_{t_0}}),
\]
or, equivalently,
\[ 
\prod_{j=0}^{k-1}\mathbf{1}_{\R_+}\left(y^{\epsilon}_{t_{j}}\right) = \mathbf{1}_{\R_+}\left(y^{\epsilon}_{t_0}\right)\cdot\mathbf{1}_{(t_{k-1},+\infty)}\left(\tau^{\delta,\epsilon}_{y_{t_0}}\right)
\]
Then, the expectation in \eqref{eq: reduction3 OU} can be written as
\begin{eqnarray}
\nonumber\mathbb{E} \left(y^{\epsilon}_{t_{0}}y^{\epsilon}_{t_{k-1}} \prod_{j=1}^k\mathbf{1}_{\R_+}\left(y^{\epsilon}_{t_{j-1}}\right)\right) = \mathbb{E} \left(y^{\epsilon}_{t_{0}}\mathbf{1}_{\R_+}\left(y^{\epsilon}_{t_0}\right)\cdot y^{\epsilon}_{t_{k-1}}\mathbf{1}_{(t_{k-1},+\infty)}\left(\tau^{\delta,\epsilon}_{y_{t_0}}\right)\right) \\
= \mathbb{E} \left(y^{\epsilon}_{t_{0}}\mathbf{1}_{\R_+}\left(y^{\epsilon}_{t_0}\right)\cdot y^{\epsilon}_{t_{k-1}}\right) - \mathbb{E} \left(y^{\epsilon}_{t_{0}}\mathbf{1}_{\R_+}\left(y^{\epsilon}_{t_0}\right)\cdot y^{\epsilon}_{t_{k-1}}\mathbf{1}_{[0,t_{k-1}]}\left(\tau^{\delta,\epsilon}_{y_{t_0}}\right)\right).
\label{reduction 3a OU}
\end{eqnarray}
The first expectation above can be written as
\begin{eqnarray*}
\mathbb{E} \left(y^{\epsilon}_{t_{0}}\mathbf{1}_{\R_+}\left(y^{\epsilon}_{t_0}\right)\cdot y^{\epsilon}_{t_{k-1}}\right) &=& \mathbb{E} \left(y^{\epsilon}_{t_{0}}\mathbf{1}_{\R_+}\left(y^{\epsilon}_{t_0}\right)\mathbb{E}\left( y^{\epsilon}_{t_{k-1}}| y_{t_0}\right)\right)\\ 
&=& e^{-\frac{t_{k-1}}{\epsilon^2}}\mathbb{E} \left((y^{\epsilon}_{t_{0}})^2\mathbf{1}_{\R_+}\left(y^{\epsilon}_{t_0}\right)\right) = \frac{1}{4}e^{-\frac{t_{k-1}}{\epsilon^2}},
\end{eqnarray*}
since $y_{t_0}^{\epsilon}\sim\mathcal{N}\left(0,\frac{1}{2}\right)$ (we have assumed stationarity). Remember that $T=1$, $t_k = k\delta$ and $\delta = \frac{1}{n}$. Then, using the above result, we get that 
\begin{eqnarray*}
4\epsilon^{2}(1-e^{-\frac{\delta}{\epsilon^2}})^{2}\sum_{k=2}^{n}(n+1-k)\mathbb{E} \left(y^{\epsilon}_{t_{0}}\mathbf{1}_{\R_+}\left(y^{\epsilon}_{t_0}\right)\cdot y^{\epsilon}_{t_{k-1}}\right)  = \\
= \epsilon^{2}(1-e^{-\frac{1}{n\epsilon^2}})^{2}\sum_{k=2}^{n}(n+1-k)e^{-\frac{(k-1)}{n\epsilon^2}} \to 1+\epsilon^{2}\left(e^{-1/\epsilon^{2}}-1\right) = 1 + {\mathcal O}(\epsilon^2),
\end{eqnarray*}
as $n\to\infty$. Thus,
\[ \lim_{\epsilon\rightarrow0}\lim_{n\rightarrow\infty}4\epsilon^{2}(1-e^{-\frac{\delta}{\epsilon^2}})^{2}\sum_{k=2}^{n}(n+1-k)\mathbb{E} \left(y^{\epsilon}_{t_{0}}\mathbf{1}_{\R_+}\left(y^{\epsilon}_{t_0}\right)\cdot y^{\epsilon}_{t_{k-1}}\right)= 1.\]
From this result and the decomposition of the expectation given in \eqref{reduction 3a OU}, it follows that to prove \eqref{eq: reduction3 OU}, it is only left to show that
\[ 
\lim_{\epsilon\rightarrow0}\lim_{n\rightarrow\infty}\left|4\epsilon^{2}(1-e^{-\frac{\delta}{\epsilon^2}})^{2}\sum_{k=2}^{n}(n+1-k)\mathbb{E} \left(y^{\epsilon}_{t_{0}}\mathbf{1}_{\R_+}\left(y^{\epsilon}_{t_0}\right)\cdot y^{\epsilon}_{t_{k-1}}\mathbf{1}_{[0,t_{k-1}]}\left(\tau^{\delta,\epsilon}_{y_{t_0}}\right)\right)\right|= 0.
\]
We can restrict our study to the case where the stopping time $\tau^{\delta,\epsilon}_{y_{t_0}}$ is bounded by $t_{k-1}$, i.e. $\tau^{\delta,\epsilon}_{y_{t_0}}\leq t_{k-1}$. Then
\[ \mathbb{E}\left( y^\epsilon_{t_{k-1}}|\tau^{\delta,\epsilon}_{y_{t_0}}\right) = y^\epsilon_{\tau^{\delta,\epsilon}_{y_{t_0}}}e^{-\frac{t_{k-1}-\tau^{\delta,\epsilon}_{y_{t_0}}}{\epsilon^2}}.
\]
Thus,
\begin{eqnarray*}
\left|4\epsilon^{2}(1-e^{-\frac{\delta}{\epsilon^2}})^{2}\sum_{k=2}^{n}(n+1-k)\mathbb{E} \left(y^{\epsilon}_{t_{0}}\mathbf{1}_{\R_+}\left(y^{\epsilon}_{t_0}\right)\cdot y^{\epsilon}_{t_{k-1}}\mathbf{1}_{[0,t_{k-1}]}\left(\tau^{\delta,\epsilon}_{y_{t_0}}\right)\right)\right| \leq\\
4\epsilon^{2}(1-e^{-\frac{\delta}{\epsilon^2}})^{2}\sum_{k=2}^{n}(n+1-k)\mathbb{E} \left(y^{\epsilon}_{t_{0}}\mathbf{1}_{\R_+}\left(y^{\epsilon}_{t_0}\right) |y^\epsilon_{\tau^{\delta,\epsilon}_{y_{t_0}}}|\right) \leq \\
4\epsilon^{2}(1-e^{-\frac{\delta}{\epsilon^2}})^{2}\sum_{k=2}^{n}(n+1-k)\mathbb{E} \left((y^{\epsilon}_{t_{0}})^2\mathbf{1}_{\R_+}\left(y^{\epsilon}_{t_0}\right)\right)^\frac{1}{2}\mathbb{E}\left((y^\epsilon_{\tau^{\delta,\epsilon}_{y_{t_0}}})^2\right)^\frac{1}{2} = \\
= \left( 2\epsilon^{2}(1-e^{-\frac{\delta}{\epsilon^2}})^{2}\sum_{k=2}^{n}(n+1-k)\right)\mathbb{E}\left((y^\epsilon_{\tau^{\delta,\epsilon}_{y_{t_0}}})^2\right)^\frac{1}{2}.
\end{eqnarray*}
First, a straight forward computation gives
\[ \lim_{n\to\infty} 2\epsilon^{2}(1-e^{-\frac{\delta}{\epsilon^2}})^{2}\sum_{k=2}^{n}(n+1-k) = \frac{1}{\epsilon^2}.
\]
It is left to show that $\lim_{n\to\infty} \mathbb{E}\left((y^\epsilon_{\tau^{\delta,\epsilon}_{y_{t_0}}})^2\right) = 0$. This follows directly from the continuity of the diffusion paths with probability 1. First, it follows that $\tau^{\delta,\epsilon}_{y_{t_0}}\to \tau^{\epsilon}_{y_{t_0}}$ as $\delta\to 0$ ($n\to\infty$) with probability 1, where $\tau^{\epsilon}_{y_{t_0}}$ is the first time that the continuous process becomes zero. Then, it follows that $y^\epsilon_{\tau^{\delta,\epsilon}_{y_{t_0}}} \to y^\epsilon_{\tau^{\epsilon}_{y_{t_0}}}$ and by definition, $y^\epsilon_{\tau^{\epsilon}_{y_{t_0}}} = 0$, as $n\to\infty$, with probability 1. 

\end{proof}


%% file: General.tex
\section{Main Result}
\label{sec: general}

In the section, we extend the results of section \ref{sec: OU} to the general setting of section \ref{sec: setting}. More precisely, we consider the fast/slow systems of SDEs described in \eqref{eq:general_model}, where $\left(x^{\epsilon},y^{\epsilon}\right)\in\R\times\R^d$, $V$ is the standard $d$--dimensional Browinan motion. Moreover, we will assume the following 

\begin{ass}
\label{ass: general}
Functions $f(\cdot)$, $g(\cdot)$ and $\beta(\cdot)$ are such that the following hold
\begin{itemize}
\item[1.] Ergodicity: $y^\epsilon$ is an ergodic process with invariant distribution $\rho^\infty$.
\item[2.] Stationarity:  $y^\epsilon_0$ distributed according to the invariant distribution $\rho^\infty$.
\item[3.] The diffusion operator $\mathcal{L}$ corresponding to \eqref{eq:general_fast_variable} for $\epsilon=1$ is essentially self-adjoined, which implies that the transition semigroup will be bounded operators, contracting in $L_2$.
\item[4.] Function $f(\cdot)$ is square-integrable with respect to $\rho^\infty$ and twice differentiable.
\item[5.] Centering condition:
\begin{equation}
\mathbb{E}\left( f(y^\epsilon_0) \right) = \int_{\R^d}f(y)\rho^{\infty}(y)dy=0.\label{eq:cenetring_conditon}
\end{equation}
\end{itemize}
\end{ass}
\noindent Assumptions $3$ and $4$ above are needed, so that the functions on which the semigroup operators act are in the domain of the operators.
Assumption $5$ is necessary for the homogenization limit to exist. Under this assumptions, it is a well-known result (see \cite{pavliotis2008multiscale}) that as $\epsilon\rightarrow0$, the process $x^{\epsilon}$ converges weakly to the process $X$ solving the following SDE
\begin{equation}
dX_t=\sigma dW_t,\qquad X_0=x_0,\label{eq:Homo_general_model}
\end{equation}
where $W$ is a standard Brownian motion. The diffusion coefficient $\sigma$, which is constant for this class of models, is given by
\begin{equation}
\sigma^{2}=2\int_{\R^d}f(y)\Phi(y)\rho^{\infty}(y)dy=2\mathbb{E}\left[f(y^{\epsilon}_0)\Phi(y^{\epsilon}_0)\right].
\end{equation}
Function $\Phi(\cdot)$ above is the solution to the Poisson problem
\begin{eqnarray}
\left(\mathcal{L}\Phi\right)(y) & = & -f(y),\nonumber \\
\int_{\R^d}\Phi(y)\rho^{\infty}(y)dy & = & 0.\label{eq:Poisson_problem}
\end{eqnarray}
Notice that for this particular model, the homogenized SDE \eqref{eq:Homo_general_model} does not contain a drift coefficient. 

As before, our goal is to find an efficient estimator for $\sigma^2$, given $x^\epsilon$. We will prove that the (ExtQV) estimator defined in \ref{subsec: ExtQV} for $(x^\epsilon,y^\epsilon)$ satisfying \eqref{eq:general_model} is an asymptotically unbiased estimator of the diffusion coefficient of the limiting diffusion process $\sigma^2$. More precisely, we will prove the following

\begin{thm}
\label{thm: general_model_thm}
Let $x^{\epsilon}:[0,T]\rightarrow\mathbb{R}$ be a real\textendash valued
path described by \eqref{eq:general_model}. Then, subject to technical assumptions given in \ref{ass: general} and assumption \eqref{ass: lemma_general} of lemma \ref{lemma: general}, the following holds:
\begin{equation}
\lim_{\epsilon\rightarrow0}\lim_{n\rightarrow\infty}\mathbb{E}\left[\left(D_{2}^{\text{Ext}}(x)_{n}\right)\right]=2\mathbb{E}\left[f(y^{\epsilon}_0)\Phi(y^{\epsilon}_0)\right].
\end{equation}
\end{thm}

First, we prove the following

\begin{lem}
\label{lemma: general}
Suppose that $(x^\epsilon,y^\epsilon)$ satisfy \eqref{eq:general_model}. As before, let ${\mathcal D}_n = \{k\delta, k=0,\dots,n\}$ be the homogeneous grid of $[0,T]$, for $T=1$ and $\delta = \frac{1}{n}$ . Then, we can write
\[ \Delta x^\epsilon_{t_i} := x^\epsilon_{t_i} - x^\epsilon_{t_{i-1}} = \frac{\delta}{\epsilon} f(y^\epsilon_{t_{i-1}}) + R_{t_{i-1},t_i}(y^\epsilon),\]
where
\[ {\mathbb E}\left| R^\epsilon_{t_{i-1},t_i}(y^\epsilon)\right | \leq C \delta^{\frac{3}{2}}, \]
for some constant $C>0$ depending on $f,g,\beta$ and $\epsilon$, assuming that
\begin{equation}
\label{ass: lemma_general}
{\mathbb E}\left| A(y^\epsilon_t) \right| \leq C_A\ \ {\rm and}\ \ {\mathbb E}\left( B(y^\epsilon_t)^2 \right) \leq C_B
\end{equation}
uniformly on $t\in[0,1]$, where $A:\R^d\to\R$ and $B:\R^d\to\R$ are given by
\[ A(y) = \triangledown f(y)^* g(y) + \frac{1}{2}{\rm Tr}\left( \beta(y)^* H(f)(y) \beta(y) \right)\]
and
\[ B(y) = \triangledown f(y)^* g(y),\]
for ${z^*}$ denoting the transpose of any vector $z\in\R^d$, ${\rm Tr}\left(\cdot\right)$ denoting the trace of a matrix and $H(f)$ denoting the Hessian of the function $f:\R^d\to\R$.
\end{lem}
\begin{proof}
First, we write
\begin{eqnarray*}
\Delta x^\epsilon_{t_i} &=& \int_{t_{i-1}}^{t_i} dx^\epsilon_u = \int_{t_{i-1}}^{t_i} \frac{1}{\epsilon} f(y^\epsilon_u) du =  \int_{t_{i-1}}^{t_i} \frac{1}{\epsilon} \left( f(y^\epsilon_{t_i}) + \int_{t_{i-1}}^u df(y^\epsilon_s) \right) du \\
&=& \frac{\delta}{\epsilon} f(y^\epsilon_{t_{i-1}}) + R_{t_{i-1},t_i}(y^\epsilon),
\end{eqnarray*}
where
\[ R_{t_{i-1},t_i}(y^\epsilon) = \frac{1}{\epsilon}\int_{t_{i-1}}^{t_i} \int_{t_{i-1}}^u df(y^\epsilon_s) du
\]
Using It\^o's formula for $df(y^\epsilon_s)$, $R_{t_{i-1},t_i}(y^\epsilon)$ can be written as
\[ R_{t_{i-1},t_i} =\frac{1}{\epsilon}\int_{t_{i-1}}^{t_i} \int_{t_{i-1}}^u \left(\frac{1}{\epsilon^2}A(y^\epsilon_s) ds + \frac{1}{\epsilon} B(y^\epsilon_s) dV_s \right).
\]
The result follows from \eqref{ass: lemma_general} and It\^o's isometry.
\end{proof}

Now, we are ready to prove the theorem, following the same steps as the proof of theorem \ref{thm: main_theorem_OUmodel}. 
\begin{proof}
Following exactly the same arguments as in theorem \ref{thm: main_theorem_OUmodel}, we can show that it is sufficient to prove that 
\begin{equation}
\label{eq: reduction1 general} 
\lim_{\epsilon\rightarrow0}\lim_{n\rightarrow\infty}2\sum_{k=2}^{n}(n+1-k) \mathbb{E}\left( \Delta x^{\epsilon}_{t_{1}}\Delta x^{\epsilon}_{t_{k}}\prod_{j=1}^{k-1}\mathbf{1}_{\R_+}\left(\Delta x^{\epsilon}_{t_{j}}\Delta x^{\epsilon}_{t_{j+1}}\right)\right)
= 2\mathbb{E}\left[f(y^{\epsilon}_0)\Phi(y^{\epsilon}_0)\right].
\end{equation}
Note that, while we can still assume stationarity of increments, we cannot assume symmetry, which is why we now need to consider both cases of all positive or all negative increments. Again, using the same arguments as in theorem \ref{thm: main_theorem_OUmodel} in conjunction with lemma \ref{lemma: general}, \eqref{eq: reduction1 general} can be further reduced to 
\begin{equation}
\label{eq: reduction2 general} 
\lim_{\epsilon\rightarrow0}\lim_{n\rightarrow\infty}\sum_{k=2}^{n}\frac{n+1-k}{n^2\epsilon^2} \mathbb{E}\left( f(y^{\epsilon}_{t_{0}})f(y^{\epsilon}_{t_{k-1}})\prod_{j=1}^{k-1}\mathbf{1}_{\R_+}\left(f(y^{\epsilon}_{t_{j-1}})f(y^{\epsilon}_{t_{j}})\right)\right)
= \mathbb{E}\left[f(y^{\epsilon}_0)\Phi(y^{\epsilon}_0)\right].
\end{equation}
Let $\tau^{\delta,\epsilon}$ be the first time that the discretised process $\{ f(y^\epsilon_{t_k}); k\in\N \}$ changes sign, i.e.
\begin{equation}
\label{eq: tau_general}
\tau^{\delta,\epsilon} = \min\{ k\in\N; f(y^\epsilon_{t_0})f(y^\epsilon_{t_k}) \leq0\}.
\end{equation}
Then, the event of ``all $\{f(y^\epsilon_{t_i}), i=0,\dots,k-1\}$ have the same sign'' is the same as the event ``time when process $f(y^\epsilon_{t_i})$ changes sign is greater or equal to $t_k$'', i.e.
\[ \prod_{j=1}^{k-1}\mathbf{1}_{\R_+}\left(f(y^{\epsilon}_{t_{j-1}})f(y^{\epsilon}_{t_{j}})\right) = \mathbf{1}_{[t_k,+\infty)}(\tau^{\delta, \epsilon}).\]
Thus, we can write
\begin{eqnarray}
\label{eq: stopping time in expectation}
\nonumber\mathbb{E}\left( f(y^{\epsilon}_{t_{0}})f(y^{\epsilon}_{t_{k-1}})\prod_{j=1}^{k-1}\mathbf{1}_{\R_+}\left(f(y^{\epsilon}_{t_{j-1}})f(y^{\epsilon}_{t_{j}})\right)\right) = \mathbb{E}\left( f(y^{\epsilon}_{t_{0}})f(y^{\epsilon}_{t_{k-1}})\mathbf{1}_{[t_k,+\infty)}(\tau^{\delta, \epsilon})\right)\\ 
=  \mathbb{E}\left( f(y^{\epsilon}_{t_{0}})f(y^{\epsilon}_{t_{k-1}})\right) - \mathbb{E}\left( f(y^{\epsilon}_{t_{0}})f(y^{\epsilon}_{t_{k-1}})\mathbf{1}_{[0,t_k)}(\tau^{\delta, \epsilon})\right).
\end{eqnarray}
By replacing \eqref{eq: stopping time in expectation} into \eqref{eq: reduction2 general}, the proof is further reduced to proving 
\begin{equation}
\label{eq: reduction3_general}
\lim_{\epsilon\rightarrow0}\lim_{n\rightarrow\infty}\sum_{k=2}^{n}\frac{n+1-k}{n^2\epsilon^2} \mathbb{E}\left( f(y^{\epsilon}_{t_{0}})f(y^{\epsilon}_{t_{k-1}})\right) = \mathbb{E}\left[f(y^{\epsilon}_0)\Phi(y^{\epsilon}_0)\right]
\end{equation}
and
\begin{equation}
\label{eq: reduction4_general}
\lim_{\epsilon\rightarrow0}\lim_{n\rightarrow\infty}\sum_{k=2}^{n}\frac{n+1-k}{n^2\epsilon^2} \mathbb{E}\left( f(y^{\epsilon}_{t_{0}})f(y^{\epsilon}_{t_{k-1}})\mathbf{1}_{[0,t_k)}(\tau^{\delta, \epsilon})\right) = 0.
\end{equation}
To complete the proof, we need to write the conditional expectation in terms of the solution of the backward Kolmogorov equation expressed in terms of the semi-group generated by the diffusion operator ${\mathcal L}$ corresponding to \eqref{eq:general_fast_variable}, scaled to $\epsilon = 1$, i.e. 
\begin{equation}
\label{eq: conditional exp_general}
{\mathbb E}\left( f(y^\epsilon_t )|y^\epsilon_s\right) = (e^\frac{-{\mathcal L}(t-s)}{\epsilon^2}f)(y^\epsilon_s).
\end{equation}
This is well defined, based on assumptions 3 and 4 of \eqref{ass: general}. To prove \eqref{eq: reduction4_general}, we first note that the above formula will still hold when we condition on a stopping time, or, more precisely, 
\[ 
{\mathbb E}\left( f(y^\epsilon_t)|y^\epsilon_{\tau^{\delta, \epsilon}}\right) = (e^\frac{-{\mathcal L}(t-s)}{\epsilon^2}f)(y^\epsilon_{\tau^{\delta,\epsilon}}).
\]
We also note that the semigroup is a contraction in $L_2$, i.e.
\[ 
{\mathbb E}\left( f(y^\epsilon_t)^2 \right) \leq {\mathbb E}\left( f(y^\epsilon_{\tau^{\delta,\epsilon}})^2 \right),
\]
for $\tau^{\delta,\epsilon}<t$. Thus,
\begin{eqnarray*}
\left| \mathbb{E}\left( f(y^{\epsilon}_{t_{0}})\mathbf{1}_{[0,t_{k}))}(\tau^{\delta, \epsilon})f(y^{\epsilon}_{t_{k-1}})\right) \right|&\leq& \mathbb{E}\left( |f(y^{\epsilon}_{t_{0}})| \cdot |{\mathbb E}\left( f(y^{\epsilon}_{t_{k-1}}) | \tau^{\delta,\epsilon}\right)|\right) \\
& \leq & \mathbb{E}\left( | f(y^{\epsilon}_{t_{0}}) | \cdot |f(y^{\epsilon}_{\tau^{\delta,\epsilon}})|\right) \\
&\leq& \mathbb{E}\left( f(y^{\epsilon}_{t_{0}})^2\right)^\frac{1}{2} \mathbb{E}\left( f(y^{\epsilon}_{\tau^{\delta,\epsilon}})^2\right)^\frac{1}{2},
\end{eqnarray*}
and, consequently, 
\begin{equation*}
\left| \sum_{k=2}^{n}\frac{n+1-k}{n^2\epsilon^2} \mathbb{E}\left( f(y^{\epsilon}_{t_{0}})f(y^{\epsilon}_{t_{k-1}})\mathbf{1}_{[0,t_k)}(\tau^{\delta, \epsilon})\right)\right| \leq \left( \sum_{k=2}^{n}\frac{n+1-k}{n^2\epsilon^2}\right)\mathbb{E}\left( f(y^{\epsilon}_{t_{0}})^2\right)^\frac{1}{2} \mathbb{E}\left( f(y^{\epsilon}_{\tau^{\delta,\epsilon}})^2\right)^\frac{1}{2}
\end{equation*}
It is easy to see that
\[ 
\lim_{n\to\infty} \left( \sum_{k=2}^{n}\frac{n+1-k}{n^2\epsilon^2}\right)\mathbb{E}\left( f(y^{\epsilon}_{t_{0}})^2\right)^\frac{1}{2} = C(\epsilon)<\infty,
\]
and thus it remains to show that
\[
\lim_{n\to\infty} \mathbb{E}\left( f(y^{\epsilon}_{\tau^{\delta,\epsilon}})^2\right) = 0.
\]
This follows from the almost sure continuity of the paths $f(y_t)$, which implies that $\tau^{\delta,\epsilon}\to \tau^\epsilon$ with probability 1 as $\delta\to 0$ (or $n\to\infty$), where $\tau^\epsilon = \min\{t>0: f(y_t) = 0\}$. Continuity also implies that $f(y_{\tau^{\delta,\epsilon}})\to f(y_{\tau^\epsilon})$ with probability 1 as $\delta\to 0$, and thus
\[ 
\lim_{n\to\infty}\mathbb{E}\left( f(y^{\epsilon}_{\tau^{\delta,\epsilon}})^2\right) = \mathbb{E}\left( f(y^{\epsilon}_{\tau^{\epsilon}})^2\right) = 0,
\]
since $f(y^{\epsilon}_{\tau^{\epsilon}}) = 0$, by the definition of $\tau^\epsilon$. 

Finally, it remains to show \eqref{eq: reduction3_general}. Using \eqref{eq: conditional exp_general}, we write
\[ 
\mathbb{E}\left( f(y^{\epsilon}_{t_{0}})f(y^{\epsilon}_{t_{k-1}})\right)  = \mathbb{E}\left( f(y^{\epsilon}_{t_{0}})\mathbb{E}\left(f(y^{\epsilon}_{t_{k-1}})| y_{t_0}\right)\right) = \mathbb{E}\left( f(y^{\epsilon}_{t_{0}})(e^\frac{-\mathcal{L}(k-1)}{n\epsilon^2}f)(y^{\epsilon}_{0})\right), 
\]
Using the dominated convergence theorem, we can write the left-hand-side of \eqref{eq: reduction3_general} as
\begin{eqnarray*}
\lim_{\epsilon\rightarrow0}\lim_{n\rightarrow\infty}\sum_{k=2}^{n}\frac{n+1-k}{n^2\epsilon^2} \mathbb{E}\left( f(y^{\epsilon}_{t_{0}})(e^\frac{-\mathcal{L}(k-1)}{n\epsilon^2}f)(y^{\epsilon}_{0})\right) = \\
\ \ = \mathbb{E}\left( f(y^{\epsilon}_{t_{0}})\lim_{\epsilon\rightarrow0}\lim_{n\rightarrow\infty}\sum_{k=2}^{n}\frac{n+1-k}{n^2\epsilon^2} (e^\frac{-\mathcal{L}(k-1)}{n\epsilon^2}f)(y^{\epsilon}_{0})\right).
\end{eqnarray*}
Let $\Psi:\R^d\to\R$ be defined as
\[ 
\Psi(y) = \lim_{\epsilon\rightarrow0}\lim_{n\rightarrow\infty}\sum_{k=2}^{n}\frac{n+1-k}{n^2\epsilon^2} (e^\frac{-\mathcal{L}(k-1)}{n\epsilon^2}f)(y) = \lim_{\epsilon\rightarrow0}\lim_{n\rightarrow\infty}\sum_{k=1}^{n-1}\frac{n-k}{n^2\epsilon^2} (e^\frac{-\mathcal{L}k}{n\epsilon^2}f)(y)
\]
If we can show that $\Psi$ solves the Poisson problem \eqref{eq:Poisson_problem}, then the proof would be complete. First, we note that 
\begin{eqnarray*}
\mathbb{E}\left( \Psi(y_0)\right)&=& \mathbb{E}\left( \lim_{\epsilon\rightarrow0}\lim_{n\rightarrow\infty}\sum_{k=1}^{n-1}\frac{n-k}{n^2\epsilon^2} (e^\frac{-\mathcal{L}k}{n\epsilon^2}f)(y_0)\right) \\
&=&  \lim_{\epsilon\rightarrow0}\lim_{n\rightarrow\infty}\sum_{k=1}^{n-1}\frac{n-k}{n^2\epsilon^2}\mathbb{E}\left( (e^\frac{-\mathcal{L}k}{n\epsilon^2}f)(y_0)\right) \\
&=& \lim_{\epsilon\rightarrow0}\lim_{n\rightarrow\infty}\sum_{k=1}^{n-1}\frac{n-k}{n^2\epsilon^2}\mathbb{E}\left( f(y^\epsilon_{t_k}) \right) \\
&=& \lim_{\epsilon\rightarrow0}\lim_{n\rightarrow\infty}\sum_{k=1}^{n-1}\frac{n-k}{n^2\epsilon^2}\mathbb{E}\left( f(y^\epsilon_{0}) \right) = 0,
\end{eqnarray*}
where we used the dominated convergence theorem to go from the first line to the second, the tower property and \eqref{eq: conditional exp_general} (for $s=0$, $t = t_k$) to go from second line to the third and stationarity to go from the third line to the fourth. Finally, we write
\begin{eqnarray*}
\Psi &=&  \lim_{\epsilon\rightarrow0}\lim_{n\rightarrow\infty}\sum_{k=1}^{n-1}\frac{n-k}{n^2\epsilon^2} e^\frac{-\mathcal{L}k}{n\epsilon^2}f  \\
&=& \lim_{\epsilon\rightarrow0}\lim_{n\rightarrow\infty}\sum_{k=1}^{n-1}\frac{n-k}{n^2\epsilon^2}\sum_{m=0}^\infty \frac{1}{m!}(\frac{k}{n\epsilon^2})^m (\mathcal{-L})^m f \\
&=& \lim_{\epsilon\rightarrow0}\sum_{m=0}^\infty \frac{1}{m!}\left( \lim_{n\rightarrow\infty}\sum_{k=1}^{n-1}\frac{n-k}{n^2\epsilon^2} (\frac{k}{n\epsilon^2})^m\right) (\mathcal{-L})^m f \\
&=&  \lim_{\epsilon\rightarrow0}\sum_{m=0}^\infty \frac{1}{m!} \frac{1}{(m+1)(m+2)\epsilon^{2(m+1)}} (\mathcal{-L})^m f \\
&=&  \lim_{\epsilon\rightarrow0}\epsilon^2 \sum_{m=0}^\infty \frac{1}{(m+2)!} \left(\frac{1}{\epsilon^2}\right)^{(m+2)} (\mathcal{-L})^m f,
\end{eqnarray*}
where we used the definition of the exponential of an operator to go from the first line to the second, the fact that limits exists to go from the second line to the third and the appropriate limit identity to go from the third line to the fourth. Finally,
\begin{eqnarray*}
\mathcal{L}^2 \Psi &=&  \mathcal{L}^2 \left( \lim_{\epsilon\rightarrow0}\epsilon^2 \sum_{m=0}^\infty \frac{1}{(m+2)!} \left(\frac{1}{\epsilon^2}\right)^{(m+2)} (\mathcal{-L})^m \right) f \\
&=& \lim_{\epsilon\rightarrow0}\epsilon^2 \sum_{m=0}^\infty \frac{1}{(m+2)!} \left(\frac{1}{\epsilon^2}\right)^{(m+2)} (\mathcal{-L})^{(m+2)} f \\
&=& \lim_{\epsilon\rightarrow0}\epsilon^2 \sum_{m=2}^\infty \frac{1}{m!} \left(\frac{1}{\epsilon^2}\right)^{m} (\mathcal{-L})^{m} f \\
&=& \lim_{\epsilon\rightarrow0}\epsilon^2 \left( \sum_{m=0}^\infty \frac{1}{m!} \left(-\frac{1}{\epsilon^2}\right)^{m} \mathcal{L}^{m} - I - \frac{1}{\epsilon^2} \mathcal{L}\right) f \\
&=& \lim_{\epsilon\rightarrow0}\epsilon^2 \left( e^\frac{\mathcal{-L}}{\epsilon^2}- I\right)f  - \mathcal{L} f,
\end{eqnarray*}
where we used the continuity of the operator, which allows us to apply it before the limits. Using the contraction property, we can show that
\[ \lim_{\epsilon\rightarrow0}\epsilon^2 \left( e^\frac{\mathcal{-L}}{\epsilon^2}- I\right) f = 0\]
in $L_2$. Thus, 
\[ \mathcal{L}^2 \Psi = -\mathcal{L}f,\]
which is also satisfied by $\Phi$. We conclude that $\Psi = \Phi$, as solution is unique.
\end{proof}

%% file: numerics.tex
\section{Numerical Results}
\label{sec: numerical_results}

In this section we present numerical results for the performance of the (ExtQV) estimator when it is applied to different examples of multiscale models.

\begin{example}
\label{ex: simple_OU}
We start our numerical study  from the the toy example that was introduced in \eqref{subeq: multiOU}. Initially, we present numerical evidence supporting the unbiasedness of our proposed estimator.  We also examine how the choice of the parameter $\epsilon$, the step size $\delta=T/n$ and the value of $\sigma$ affects the accuracy of the (ExtQV). Unless stated otherwise, $T=1$. 

We generate 1000 realisations of the path $x^{\epsilon}$ with step $\delta = \frac{1}{n}$, using the Euler--Maruyama scheme. For each realisation we evaluate the (ExtQV). We approximate the expectation by the average of these values. 

Table \ref{tab:ExtQV-vs} presents the values of the expectation of the (ExtQV) for $\epsilon=(0.05,0.10,0.15,0.20)$, $n=\left(10^{3},10^{4},10^{5},10^{6},10^{7}\right)$ and for $\sigma=1$. The last line of the table corresponds to  the theoretical value of the (ExtQV) as $n\rightarrow\infty$ which is given by
 \[
 \lim_{n\rightarrow\infty}\mathbb{E}\left[D_{2}^{\text{Ext}}(x_{n}^{\epsilon})_{T}^{2}\right]=\sigma^{2}\left(1+\epsilon^{2}\left(1-e^{-1/\epsilon^{2}}\right)\right).
 \]
\begin{table}[H]
\centering%
\begin{tabular}{|c||c|c|c|c|}
\hline
\multirow{2}{*}{$\mathbb{E}\left[D_{2}^{\text{Ext}}(x^{\epsilon})_{n}^{2}\right]$} & \multicolumn{4}{c|}{$\epsilon$}\tabularnewline
\cline{2-5}
 & 0.05 & 0.10 & 0.15 & 0.20\tabularnewline
\hline
\hline
$n=10^{3}$ & 1.4971 & 1.1956 & 1.0706 & 1.0556\tabularnewline
\hline
$n=10^{4}$ & 1.1317 & 1.0569 & 1.0327 & 0.9639\tabularnewline
\hline
$n=10^{5}$ & 1.0400 & 0.9997 & 1.0003 & 0.9682\tabularnewline
\hline
$n=10^{6}$ & 1.0085 & 0.9861 & 0.9599 & 0.9447\tabularnewline
\hline
$n=10^{7}$ & 0.9908 & 0.9904 & 0.9665 & 0.9515\tabularnewline
\hline
Theoretical Value & 0.9975 & 0.9900 & 0.9775 & 0.9600\tabularnewline\hline
\end{tabular}\caption[Expectation of the (ExtQV) for the model \eqref{subeq: multi}. Investigation of its behaviour for different $n$'s and $\epsilon$'s and for $\sigma=1$.]{\label{tab:ExtQV-vs}Expectation of the (ExtQV) for different $n$'s
and $\epsilon$'s and for $\sigma=1$.}
\end{table}

Furthermore, we examine the squared $L_2$ error of the (ExtQV), i.e.
\[
\mathbb{E}\left[\left(\mathbb{E}\left(D_{2}^{\text{Ext}}(x_{n}^{\epsilon})_{T}\right)^{2}-\sigma^{2}\right)^{2}\right].
\]
Table \ref{tab:L2-vs-n_e} shows the squared $L_{2}$--error for different $n$'s, $\epsilon$'s and for fixed $\sigma=1$. These numerical results indicate that the squared $L_2$ error of our estimator is of order ${\mathcal O}(\epsilon^{2})$. This can be seen more clearly in the log--log plot in figure \ref{fig: log_log_L2}. 
 
\begin{table}[H]
\centering%
\begin{tabular}{|c||c|c|c|c|}
\hline
\multirow{2}{*}{$\mathbb{E}\left[\left(D_{2}^{\text{Ext}}(x_{n}^{\epsilon})_{T}^{2}-\sigma^{2}\right)^{2}\right]$} & \multicolumn{4}{c|}{$\epsilon$}\tabularnewline
\cline{2-5}
 & 0.05 & 0.10 & 0.15 & 0.20\tabularnewline
\hline
\hline
$n=10^{3}$ & 0.2985 & 0.1785 & 0.1792 & 0.3638\tabularnewline
\hline
$n=10^{4}$ & 0.0476 & 0.1250 & 0.2650 & 0.3548\tabularnewline
\hline
$n=10^{5}$ & 0.0306 & 0.1154 & 0.2380 & 0.3800\tabularnewline
\hline
$n=10^{6}$ & 0.0273 & 0.10 & 0.1986 & 0.3874\tabularnewline
\hline
$n=10^{7}$ & 0.0261 & 0.1008 & 0.1992 & 0.3824\tabularnewline
\hline
\end{tabular}\caption[L$_{2}$\textendash error of the (ExtQV) for data generated by the model  \eqref{subeq: multi}. Investigation of its behavior for
different $n$'s and $\epsilon$'s and for $\sigma=1$.]{\label{tab:L2-vs-n_e}L$_{2}$\textendash error of the (ExtQV) for
different $n$'s and $\epsilon$'s and for $\sigma=1$.}
\end{table}

\begin{figure}[H]
  \centering
  \includegraphics[scale=1]{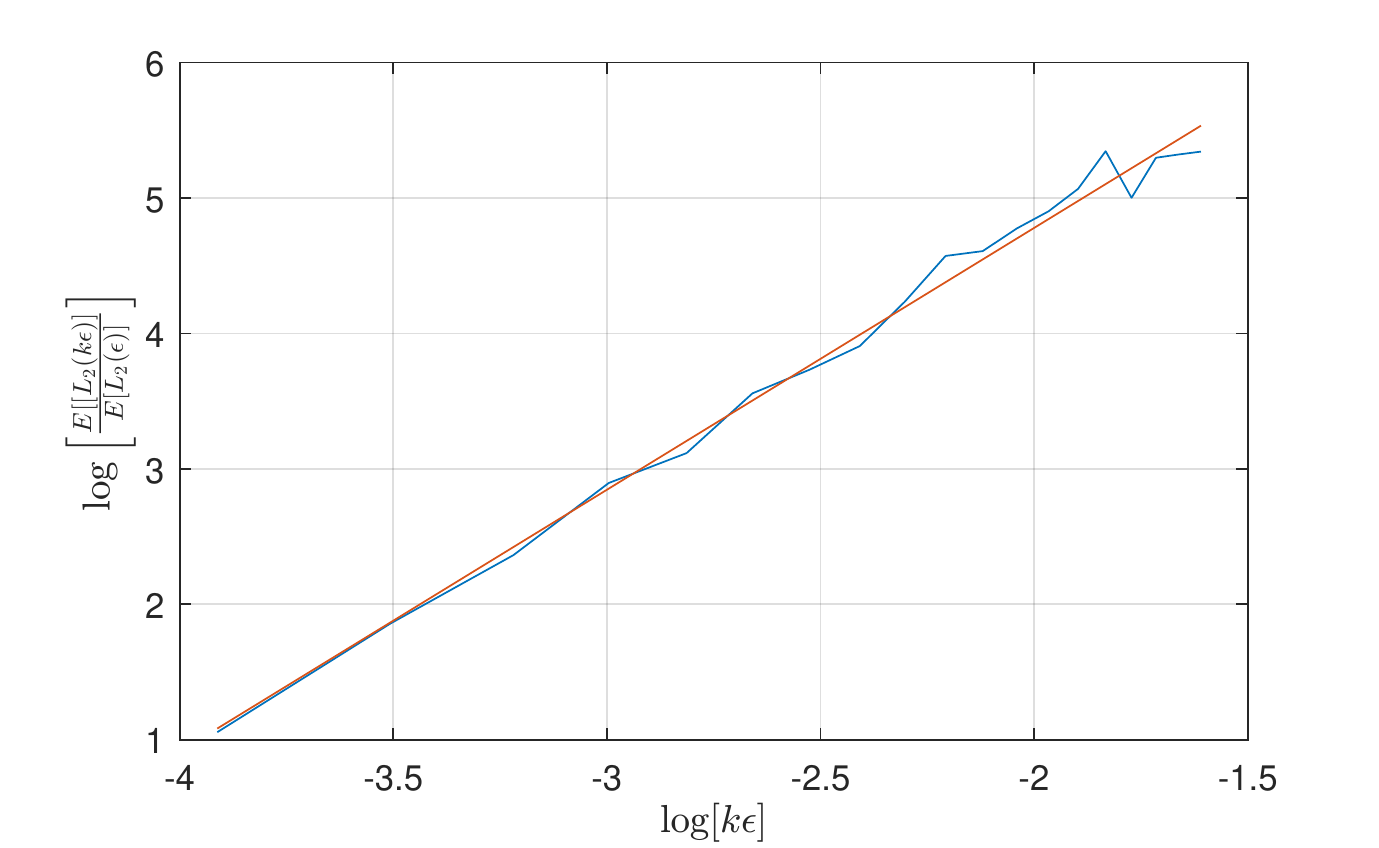}
  \caption[Investigation of the $L_{2}$--error behaviour with respect to $\epsilon$ through a log--log plot]{Log--log plot between the ration of the $L_{2}$--error corresponding to $k\epsilon$ and $ L_{2}$--error corresponding to $\epsilon$ with respect to $\log(k\epsilon)$.}\label{fig: log_log_L2}
\end{figure}

\begin{example}\label{eq: example_y_3}
Consider the following fast/slow system of SDEs
\begin{subequations}
\begin{align}
dx & =\frac{\sigma}{\epsilon}y^{3}dt,\qquad & x(0)=x_{0},\label{eq:examplel_slow_variable}\\
dy & =-\frac{y}{\epsilon^{2}}dt+\frac{\sqrt{2}}{\epsilon}dV,\qquad & y(0)=y_{0},\label{eq:example_fast_variable}
\end{align}
\label{eq:example_model}\end{subequations}where $V$ is the standard
Brownian motion and has initial conditions $x_{0}$ and $y_{0}$. The invariant density, $\rho^{\infty}$, of the fast
process in \eqref{eq:example_fast_variable} is the standard normal. 
Without loss of generality, we assume $\sigma=1$ so that the corresponding
homogenised SDE is given by
\begin{eqnarray*}
dX & = & \left(2\cdot\mathbb{E}\left[f(y)\Phi(y)\right]\right)^{1/2}dW=\sqrt{22}dW,
\label{eq:homo_cubed}
\end{eqnarray*}
where $W$ is a standard Brownian motion and is independent of $V$. 

Again, our objective is to examine the performance of the (ExtQV) estimator as an estimator of the diffusion coefficient of the homogenised equation. Table \ref{tab:example_1_model_table} shows the values of the expectation of the (ExtQV) and its corresponding $L_{2}$\textendash error when it is applied to the model \eqref{eq:example_model}. For this table we fix the value of $\sigma$ to $\sigma=0.1$ and we consider five values of $\epsilon=(0.20,0.15,0.10,0.05,0.01)$ and four values of $n=\left(10^{4},10^{5},10^{6},10^{7}\right)$. The corresponding diffusion coefficient of the homogenised equation in this case is $\Sigma^{2}=0.01\cdot22=0.22$.
As it can be seen from Table \ref{tab:example_1_model_table}, as the value of $n$ increases and the value of $\epsilon$ decreases, the expectation of the (ExtQV) tends to the real value of the homogenized diffusion coefficient. Furthermore, for decreasing $n$ and $\epsilon$ the $L_{2}$--error decreases as well and tends to zero.

\begin{table}[H]
\centering%
\begin{tabular}{|c|c||c|c|c|c|c|}
\cline{2-7}
\multicolumn{1}{c|}{} & $\epsilon$  & 0.20  & 0.15  & 0.10  & 0.05  & 0.01\tabularnewline
\hline
\hline
\multirow{2}{*}{$n=10^{4}$ } & $\mathbb{E}\left[D_{2}^{\text{Ext}}(x_{n})_{T}^{2}\right]$  & 0.1938  & 0.2087  & 0.2177  & 0.2368  & 1.6072\tabularnewline
\cline{2-7}
 & $L_{2}$-error  & 0.1260  & 0.0486  & 0.0237  & 0.0071  & 1.9342\tabularnewline
\hline
\multirow{2}{*}{$n=10^{5}$ } & $\mathbb{E}\left[D_{2}^{\text{Ext}}(x_{n})_{T}^{2}\right]$  & 0.2059  & 0.2086  & 0.2176  & 0.2281  & 0.2638\tabularnewline
\cline{2-7}
 & $L_{2}$-error  & 0.0860  & 0.0611  & 0.0257  & 0.0078  & 0.0023\tabularnewline
\hline
\multirow{2}{*}{$n=10^{6}$ } & $\mathbb{E}\left[D_{2}^{\text{Ext}}(x_{n})_{T}^{2}\right]$  & 0.1927  & 0.2047  & 0.2169  & 0.2217  & 0.2284\tabularnewline
\cline{2-7}
 & $L_{2}$-error  & 0.0630  & 0.0439  & 0.0268  & 0.0071  & 0.0004\tabularnewline
\hline
\multirow{2}{*}{$n=10^{7}$ } & $\mathbb{E}\left[D_{2}^{\text{Ext}}(x_{n})_{T}^{2}\right]$  & 0.2119  & 0.2075  & 0.211  & 0.2330  & 0.2209\tabularnewline
\cline{2-7}
 & $L_{2}$-error  & 0.0986  & 0.0434  & 0.0235  & 0.0073  & 0.0003\tabularnewline
\hline
\end{tabular}\caption[Expectation and $L_{2}$--error of the (ExtQV) for data generated by \eqref{eq:example_model}. Investigation of its behavior for different $n$'s
and $\epsilon$'s and for $\sigma=0.10$.]{\label{tab:example_1_model_table}Expectation and $L_{2}$--error of the (ExtQV) for different $\epsilon$'s, $n$'s and for $\sigma=0.10$.}
\end{table}

\end{example}

\begin{example} \label{exa: molel_example_2}Consider the following multiscale
system of SDEs
\begin{eqnarray*}
dx & = & \frac{\sigma}{\epsilon}\left(1-y^{2}\right)dt,\\
dy & = & -\frac{1}{\epsilon^{2}}ydt+\frac{\sqrt{2}}{\epsilon}dV,
\end{eqnarray*}
where $V$ is the standard Browian motion and initial conditions $x_{0}$ and $y_{0}$. For this example the corresponding homogenized SDE has
the following form
\begin{equation}
dX=\sigma\sqrt{2}dW.\label{eq:example_2_homo_SDE}
\end{equation}

Similarly to what we have done in the previous examples, we examine the (ExtQV) for four values of $n$ and five values of $\epsilon$ and the results are shown in Table \ref{tab:example_2_model_table}. 
\begin{table}[h!]
\centering%
\begin{tabular}{|c|c||c|c|c|c|c|}
\cline{2-7}
\multicolumn{1}{c|}{} & $\epsilon$  & 0.20  & 0.15  & 0.10  & 0.05  & 0.01\tabularnewline
\hline
\hline
\multirow{2}{*}{$n=10^{4}$ } & $\mathbb{E}\left[D_{2}^{\text{Ext}}(x_{n})_{T}^{2}\right]$  & 2.0426  & 2.0364  & 2.2174  & 2.3837  & 16.6118\tabularnewline
\cline{2-7}
 & $L_{2}$-error  & 5.4799  & 3.0374  & 2.1654  & 0.5906  & 213.9965\tabularnewline
\hline
\multirow{2}{*}{$n=10^{5}$} & $\mathbb{E}\left[D_{2}^{\text{Ext}}(x_{n})_{T}^{2}\right]$  & 1.9373  & 1.9752  & 2.0926  & 2.1041  & 2.6470\tabularnewline
\cline{2-7}
 & $L_{2}$-error  & 4.2892  & 3.2262  & 1.4939  & 0.3571  & 0.4380\tabularnewline
\hline
\multirow{2}{*}{$n=10^{6}$ } & $\mathbb{E}\left[D_{2}^{\text{Ext}}(x_{n})_{T}^{2}\right]$  & 1.9120  & 2.0136  & 1.9397  & 2.0416  & 2.1744\tabularnewline
\cline{2-7}
 & $L_{2}$-error  & 4.9281  & 2.9718  & 1.2007  & 0.3632  & 0.0463\tabularnewline
\hline
\multirow{2}{*}{$n=10^{7}$ } & $\mathbb{E}\left[D_{2}^{\text{Ext}}(x_{n})_{T}^{2}\right]$  & 1.9312  & 1.9721  & 2.0525  & 2.0061  & 2.0546\tabularnewline
\cline{2-7}
 & $L_{2}$-error  & 6.3127  & 3.3509  & 1.6203  & 0.3273  & 0.0176\tabularnewline
\hline
\end{tabular}\caption[Expectation and $L_{2}$--error of the (ExtQV) for data generated by the model in the Example \ref{exa: molel_example_2}. Investigation of its behavior for different $n$'s
and $\epsilon$'s and for $\sigma=1$.]{\label{tab:example_2_model_table}Expectation and $L_{2}$--error of the (ExtQV)  for different $\epsilon$'s, $n$'s and for $\sigma=1$.}
\end{table}
\end{example}

In the next example, we modify our context in the sense that the fast dynamics are not described by an (OU) process.
\begin{example}\label{ex: sin_ou_model}
Consider the following multiscale system
\begin{subequations}
\begin{eqnarray}
dx & = & \sigma\frac{\sin(y)}{\epsilon}dt,\label{eq:sin_model_slow}\\
dy & = & -\frac{\sin(y)}{\epsilon^{2}}dt+\frac{1}{\epsilon}dW,\label{eq:sin_model_fast}
\end{eqnarray}
\label{subeq: model_sin}
\end{subequations}
for which the corresponding homogenised SDE is
\begin{equation}
dX=\sigma dW.
\end{equation}

Table \ref{tab: sin_model} illustrates the expectation of the (ExtQV) and its corresponding $L_{2}$\textendash error when applied to the
model \eqref{subeq: model_sin} for $\sigma=\sqrt{0.5}$. As in the previous examples, we consider five values of $\epsilon=(0.20,0.15,0.10,0.05,0.01)$ and three values of $n=(10^{4},10^{5},10^{6})$. For $\sigma=\sqrt{0.5}$,
the corresponding homogenised diffusion coefficient is equal to 0.5. Similarly to the previous examples, we observe that as the value of
$n$ increases and the value of $\epsilon$ decreases both the expectation of the (ExtQV) and the $L_{2}$\textendash error tend to the desired quantity, that is the real value of the homogenised and coefficient and zero respectively.
\begin{table}[H]
\centering%
\begin{tabular}{|c|c|c|c|c|c|c|}
\hline
 & $\epsilon$ & 0.20 & 0.15 & 0.10 & 0.05 & 0.01\tabularnewline
\hline
\hline
\multirow{2}{*}{$n=10^{4}$} & $\mathbb{E}\left[D_{2}^{\text{Ext}}(x_{n})_{T}^{2}\right]$ & 0.3889 & 0.4023 & 0.4320 & 0.4604 & 0.8046\tabularnewline
\cline{2-7}
 & $L_{2}$\textendash error & 0.0571 & 0.0391 & 0.0198 & 0.0059 & 0.0932\tabularnewline
\hline
\multirow{2}{*}{$n=10^{5}$} & $\mathbb{E}\left[D_{2}^{\text{Ext}}(x_{n})_{T}^{2}\right]$ & 0.3707 & 0.3878 & 0.3974 & 0.4055 & 0.4336\tabularnewline
\cline{2-7}
 & $L_{2}$\textendash error & 0.0599 & 0.0408 & 0.0243 & 0.0123 & 0.0046\tabularnewline
\hline
\multirow{2}{*}{$n=10^{6}$} & $\mathbb{E}\left[D_{2}^{\text{Ext}}(x_{n})_{T}^{2}\right]$ & 0.3761 & 0.4005 & 0.4064 & 0.4201 & 0.4990\tabularnewline
\cline{2-7}
 & $L_{2}$\textendash error & 0.0577 & 0.0393 & 0.0228 & 0.0102 & 0.0002\tabularnewline
\hline
\end{tabular}
\caption{Expectation and $L_{2}$\textendash error of the (ExtQV) for different
$\epsilon$'s, $n$'s and $\sigma=\sqrt{0.5}$.\label{tab: sin_model}}
\end{table}
\end{example}

Finally, in the example below we demonstrate that our proposed estimator can be also applied in cases where the corresponding homogenised equation contains a drift term. 
\vspace{0.3cm}
\begin{example}
\label{exa:drift_model_Example}Consider the following fast/slow
system
\begin{subequations}
\begin{align}
dx & =\frac{\sigma}{\epsilon}ydt+\sin(x)dt,\label{eq:general_slow_variable_drift_example}\\
dy & =-\frac{1}{\epsilon^{2}}ydt+\frac{1}{\epsilon}dV.\label{eq:general_fast_variable_drift_example}
\end{align}
\label{eq:general_model_drift_example}\end{subequations}
 The corresponding homogenized SDE is
\begin{equation}
dX=\sin(X)dt+\sigma dW.\label{eq:Homo_general_model_drift_example}
\end{equation}
Similar numerical studies are performed for this model and Table \ref{tab:drift_model_table} shows the Expectation and $L_{2}$--error of the (ExQV) for the same values of $n$ and $\epsilon$ considered in the previous examples. The results, presented in table \ref{tab:drift_model_table}, indicate that the drift coefficient does not affect the behaviour of our proposed estimator.
\begin{table}[H]
\centering%
\begin{tabular}{|c|c||c|c|c|c|c|}
\cline{2-7}
\multicolumn{1}{c|}{} & $\epsilon$  & 0.20  & 0.15  & 0.10  & 0.05  & 0.01\tabularnewline
\hline
\hline
\multirow{2}{*}{$n=10^{4}$} & $\mathbb{E}\left[D_{2}^{\text{Ext}}(x_{n})_{T}^{2}\right]$  & 1.1429  & 1.1112  & 1.1039  & 1.1289  & 2.2738\tabularnewline
\cline{2-7}
 & $L_{2}$-error  & 0.4890  & 0.3055  & 0.1462  & 0.0505  & 1.6949\tabularnewline
\hline
\multirow{2}{*}{$n=10^{5}$} & $\mathbb{E}\left[D_{2}^{\text{Ext}}(x_{n})_{T}^{2}\right]$  & 1.1258  & 1.0709  & 1.0559  & 1.0491  & 1.2194\tabularnewline
\cline{2-7}
 & $L_{2}$-error  & 0.5250  & 0.2695  & 0.1312  & 0.0.12  & 0.0497\tabularnewline
\hline
\multirow{2}{*}{$n=10^{6}$} & $\mathbb{E}\left[D_{2}^{\text{Ext}}(x_{n})_{T}^{2}\right]$  & 1.1488  & 1.0808  & 1.0282  & 1.0446  & 1.2196\tabularnewline
\cline{2-7}
 & $L_{2}$-error  & 0.5822  & 0.2649  & 0.1115  & 0.0439  & 0.0164\tabularnewline
\hline
\multirow{2}{*}{$n=10^{7}$ } & $\mathbb{E}\left[D_{2}^{\text{Ext}}(x_{n})_{T}^{2}\right]$  & 1.1476  & 1.0728  & 1.0430  & 1.0595  & 1.2185\tabularnewline
\cline{2-7}
 & $L_{2}$-error  & 0.5661  & 0.2778  & 0.1097  & 0.0356  & 0.0167\tabularnewline
\hline
\end{tabular}\caption[Expectation and $L_{2}$--error of the (ExtQV) for data generated by the model in the Example \ref{exa:drift_model_Example}. Investigation of its behavior for different $n$'s
and $\epsilon$'s and for $\sigma=1$.]{\label{tab:drift_model_table}Expectation and $L_{2}$--error of the (ExtQV) for the model in Example \ref{exa:drift_model_Example}  for different $\epsilon$'s, $n$'s and for $\sigma=1$.}
\end{table}
\end{example}

\end{example}